\documentclass[twoside,a4paper]{amsart}

\usepackage{mathabx}

\usepackage{pdfsync}

\usepackage{xcolor}
\definecolor{webgreen}{rgb}{0,.5,0}
\definecolor{webbrown}{rgb}{.6,0,0}
\definecolor{RoyalBlue}{cmyk}{1, 0.50, 0, 0}
\usepackage[colorlinks=true, breaklinks=true, urlcolor=webbrown, linkcolor=RoyalBlue, citecolor=webgreen,backref=page]{hyperref}

\usepackage{courier} 
\normalfont
\usepackage{epsfig, graphicx, subfigure}
\usepackage{amsmath, amssymb}

\newcommand{\R}     {\mathbb{R}}
\newcommand{\C}     {\mathbb{C}}
\newcommand{\N}     {\mathbb{N}}

\newcommand{\supp}{\mathrm{supp}}

\newcommand{\df}{\stackrel{\rm def}{=}}

\def\cal{\mathcal}
\let\Re=\undefined
\DeclareMathOperator{\Re}{Re}
\let\Im=\undefined
\DeclareMathOperator{\Im}{Im}

\def\ge{\geqslant}
\def\le{\leqslant}

\newtheorem{theorem}{Theorem}[section]

\newtheorem{lemma}[theorem]{Lemma}

\theoremstyle{remark}

\numberwithin{equation}{section}

\begin{document}

\title[Generalizations of Menchov-Rademacher theorem  \ldots]{
Generalizations of Menchov-Rademacher theorem and existence of wave operators in Schr\"odinger evolution     }

\begin{abstract}
We obtain generalizations of the classical Menchov-Rademacher theorem to the case of continuous orthogonal systems. These results are applied to show the existence of Moller wave operators in Schr\"odinger evolution.
\end{abstract} \vspace{1cm}

\thanks{
The work of SD done in the first two sections
  was supported by the grant NSF-DMS-1764245 and his research on the rest of the paper was supported by the Russian Science Foundation (project RScF-19-71-30004). The work of LM was supported by the grant RTG NSF-DMS-1147523.}

\author[S. Denisov, L. Mohamed]{Sergey  Denisov, Liban Mohamed}
\address{Department of Mathematics, University of Wisconsin-Madison, 480 Lincoln Dr., Madison, WI 53706, USA}
\email{\href{mailto:denissov@wisc.edu}{denissov@wisc.edu}}

\address{Department of Mathematics, University of Wisconsin-Madison, 480 Lincoln Dr., Madison, WI 53706, USA}
\email{\href{mailto:lmohamed@wisc.edu}{lmohamed@wisc.edu}}



\maketitle

\setcounter{tocdepth}{3}

\section{Introduction}

The celebrated Menchov-Rademacher Theorem (see, e.g., \cite{ks}) gives a general condition for a.e. convergence of the orthogonal series:
\begin{theorem}[{\bf Menchov-Rademacher}]
Suppose $\{\phi_n(x)\},n\in \mathbb{N}$ is orthonormal system in $L^2(0,1)$ and the sequence $\{a_n\}$ satisfies
\[
l\df \sum_{n=1}^\infty a_n^2\log^2(n+1)<\infty\,.
\]
Then, the series $\sum_{n=1}^\infty a_n\phi_n(x)$ converges for a.e. $x\in (0,1)$. Moreover, if \[m(x)\df \sup_{n\in \N}\left|\sum_{j=1}^n a_j\phi_j(x)\right|\] defines a maximal function, then
\[
\|m\|_{L^2(0,1)}\le C l^{1/2}
\]
with some absolute constant $C$.
\end{theorem}
This result can be easily modified to cover  orthonormal systems in $L^2_\mu(0,1)$ where $\mu$ is a measure on $(0,1)$. In this paper, we prove an analog of this result for the orthogonal system with ``continuous'' parameter of orthogonality and apply it to show existence of wave operators for Schr\"odinger evolution.\smallskip

We start with the following definitions.

{\bf Definition.} We say that $f\in L^2_{\rm loc}(\mathbb{R}^+)$ if
\begin{equation}\label{c1}
\int_0^a |f(r)|^2dr<\infty
\end{equation}
for all $a>0$. 

{\bf Definition.} Let  a pair $(P,\sigma)$ consist of a function $P(r,k):\R^+\times\R\to\C$ and a measure $\sigma$ on $\R$. We say that $(P,\sigma)$ is a continuous orthonormal system if
\begin{enumerate}
\item[(a)] for $\sigma$-a.e. $k\in \R$, $P(r,k)\in L^2_{\rm loc}(\R^+)$,
\item[(b)]
for every $f\in L^2(\R^+)$ and every $a>0$, we have \[
\int_\R\left|\int_0^a f(r)P(r,k)dr\right|^2d\sigma(k)=\int_{0}^a|f(r)|^2dr\,.
\]
\end{enumerate}

Our first result is the following theorem.
\begin{theorem}\label{main0}
Suppose $(P,\sigma)$ is continuous orthonormal system and \[L\df \int_{\R^+}|f(r)|^2\log^2(2+r)dr.\] Then, the sequence $\left\{\displaystyle \int_{0}^nf(r)P(r,k)dr\right\}$ converges for $\sigma$-a.e. $k\in\R$. Moreover, if
$$M(k)\df \sup_{n\in\N}\left|\int_0^nf(r)P(r,k)dr\right|,$$
then $\|M\|_{L^2_\sigma(\R)}\leq CL^{1/2}$ with some absolute constant C.
\end{theorem}

{\bf Definition. } We will call continuous orthonormal system $(P,\sigma)$ normalized if there is a continuous positive function $\kappa$ defined on $\R$ such that
\begin{equation}\label{sd_d1}
\kappa^{-1}\in L^\infty(\R), \quad K\df \sup_{r\ge 0}\int_\R \frac{|P(r,k)|^2}{\kappa(k)}d\sigma<\infty\,.
\end{equation}

For the normalized systems, the previous theorem can be improved in the following way.

\begin{theorem}\label{main1}
Consider the normalized continuous orthonormal system $(P,\sigma,\kappa)$ and 
suppose that $f\log(2+r)\in L^2(\mathbb{R}^+)$, then  
\begin{equation}\label{sd_l2}
\int_{\mathbb{R}}\sup_{t>0}\left|  \int_{0}^{t}f(r)P(r,k)dr\right|^2\frac{d\sigma}{\kappa(k)} \lesssim (\|\kappa^{-1}\|_{L^\infty(\mathbb{R})}+K)\int_0^\infty |f(r)|^2\log^2(2+r)dr\,.
\end{equation}
Moreover, as $R\to\infty$,
\begin{equation}\label{sd_l3}
\int_0^R f(r)P(r,k)dr\to \int_0^\infty f(r)P(r,k)dr
\end{equation}
for a.e. $k$ with respect to measure $\sigma$.
\end{theorem}

 One example of  continuous orthonormal system is given by solutions $\{P(r,k)\}$ to the Krein system \cite{d1,krein}. The Krein system is the following linear system of differential equations
\begin{equation}\label{sa}
\left\{
\begin{array}{cc}
P'(r,k)=ikP(r,k)-\overline{A(r)}P_*(r,k), & P(0,k)=1\\
P_*'(r,k)=-{A}(r)P(r,k), & P_*(0,k)=1
\end{array}, \quad k\in \mathbb{C}, \quad r\ge 0\,.
\right.
\end{equation}
In this paper, we will always assume that the coefficient $A\in L^2_{\rm loc}(\R^+)$.
The Cauchy problem \eqref{sa} has the unique solution $(P(r,k),P_*(r,k))$. In \cite{krein}  (see also, e.g., \cite{den1}), Krein showed that $\{P(r,k)\}$ with $r\ge 0$ and $k\in \R$ can be viewed as continuous analogs of polynomials, orthogonal on the unit circle. In particular, there is a  measure $\sigma$ on $\R$, which satisfies
\[
\int_{\mathbb{R}}\frac{d\sigma(k)}{1+k^2}<\infty\,,
\]
and the following  property
\begin{equation}\label{ort1}
\int_{\mathbb{R}}\left|   \int_0^a f(r)P(r,k)dr\right|^2{d\sigma}=\int_0^a |f(r)|^2dr
\end{equation}
holds for every $f\in L^2(\mathbb{R}^+)$. In other words, a pair $(P,\sigma)$ gives an example of  continuous orthonormal system. Notice that \eqref{ort1} allows us to define the generalized Fourier transform
\[
\int_0^\infty f(r)P(r,k)dr
\]
as an element of $L^2_\sigma(\R)$.

Under a mild extra assumption on coefficient $A$, the system $(P,\sigma)$ becomes normalized and the previous theorem can be applied. More precisely, the following lemma holds.

\begin{lemma}\label{sd_al}Suppose the coefficient $A$ in Krein system belongs to the Stummel class, i.e.,  \begin{equation}
\|A\|_{\rm St}\df \sup_{r\ge 0}\left(\int_r^{r+1}|A(\rho)|^2d\rho\right)^{1/2}<\infty\,.
\end{equation}
Then, 
\begin{equation}\label{sd_l21}
\sup_{r>0}\int_{\R} \frac{|P(r,k)|^2}{1+k^2}d\sigma \lesssim  1+\|A\|_{\rm St}^2\,.
\end{equation}
Moreover, we have \eqref{sd_l2} and \eqref{sd_l3} with $\kappa(k)=1+k^2$ and $K\lesssim 1+\|A\|_{\rm St}^2$.
\end{lemma}
\noindent The proof of this Lemma is given in Appendix.\medskip

Another application of our general results to the Krein systems is given in the following Lemma.

\begin{lemma}\label{l1}   Suppose the coefficient in Krein system satisfies $A(r)\log(2+r)\in L^2(\mathbb{R}^+)$, then
\begin{eqnarray}\label{sd1}
\int_{\mathbb{R}} \left( \sup_{\rho<r_1<r_2} \left|\int_{r_1}^{r_2}A(x)P(x,k)dx \right|\right)^2\frac{d\sigma}{1+k^2}=\\\int_{\mathbb{R}} \left( \sup_{\rho<r_1<r_2} \left|P_*(r_2,k)-P_*(r_1,k) \right|\right)^2\frac{d\sigma}{1+k^2}\lesssim (1+\|A\|_2^2) \int_\rho^\infty |A(r)|^2\log^2(2+r)dr\,, \,\rho>0\,.\nonumber
\end{eqnarray}
Moreover, for Lebesgue a.e. $k\in \mathbb{R}$, there is a limit $\Pi(k)=\lim_{r\to\infty}P_*(r,k)$.
\end{lemma}

Theorem \ref{main0}, Theorem \ref{main1} and Lemma \ref{l1} are proved in the second section. In section 3, we apply Lemma \ref{l1} to show existence of wave operators for Schr\"odinger evolution which is our central result. Consider \[H=-\partial_{xx}^2+v\] on $\mathbb{R}^+$ with Dirichlet boundary condition at zero and denote by $H_0=-\partial_{xx}^2$ the free Schr\"odinger operator with the same Dirichlet condition at zero. The Moller wave operators (see, e.g., \cite{ya}) are defined by
\[
W^{\pm}(H,H_0)\df \lim_{t\to\pm\infty} e^{itH}e^{-itH_0}\,,
\]
where the limit is the strong limit in $L^2(\mathbb{R}^+)$. The  main result of our paper is the following theorem.
\begin{theorem}\label{t2}
Suppose $v=a'+q$ where $q\in L^1(\mathbb{R}^+)$, $a$ is absolutely continuous on $\R^+$, and
\begin{equation}
a'\in L^\infty(\mathbb{R}^+), \quad a\log(2+r)\in L^2(\mathbb{R}^+)\,.
\end{equation}
Then, the wave operators $W^{\pm}(H,H_0)$ exist.
\end{theorem}

The existence of wave and modified wave operators for Schr\"odinger and Dirac equations was extensively studied in  the scattering theory of wave propagation, see, e.g., the classical papers by Agmon \cite{a1}, H\"ormander \cite{hor}, and a book by T. Kato  \cite{kato} on the subject. The case $v\in L^p(\mathbb{R}^+), 1\le p<2$ was considered in \cite{christ} where the existence of modified wave operators was proved. See \cite{dd1} for later developments. In \cite{den1}, the presence of wave operators was established for Dirac equation with potential in $L^2(\mathbb{R}^+)$. This result is optimal on $L^p(\mathbb{R}^+)$ scale. For more general potentials in Dirac equation and connection to Szeg\H{o} condition on measure $\sigma$, see  \cite{bes}.
Some related recent results, including the multidimensional setting, can be found in, e.g., \cite{du,e,l}.

\medskip

{\bf Notation}
\begin{enumerate}
\item{If $f$ is defined on $\mathbb{R}$, $\widehat f$ denotes its Fourier transform:
\[
\widehat f(k)\df \frac{1}{\sqrt{2\pi}}\int_{\mathbb{R}}f(x)e^{- ikx}dx\,.
\]
The inverse Fourier transform is defined as 
\[
\widecheck f(k)=f^{\vee}(k)\df \frac{1}{\sqrt{2\pi}}\int_{\mathbb{R}}f(x)e^{ ikx}dx\,.
\]
} 
\item{Symbol $C^\infty (\mathbb{R})$ stands for infinitely smooth functions defined on the real line and $C_c^\infty (\mathbb{R})$ denotes the space of smooth functions with compact support.}

\item{We will use the symbol $C_{(a_1,\ldots, a_k)}$ to indicate a nonnegative function which depends on parameters $(a_1,\ldots, a_k)$. The actual value of $C$ can change from one formula to another.}

\item{If $E$ is a set on the real line, $E^c$ denotes its complement.}

\item For two non-negative functions
$f_{1(2)}$, we write $f_1\lesssim f_2$ if  there is an absolute
constant $C$ such that
\[
f_1\le Cf_2
\]
for all values of the arguments of $f_{1(2)}$. We define $\gtrsim$
similarly and say that $f_1\sim f_2$ if $f_1\lesssim f_2$ and
$f_2\lesssim f_1$ simultaneously.

\item If $f_2$ is non-negative function and $|f_1|\lesssim f_2$, we write $f_1=O(f_2)$.

\end{enumerate}\bigskip

\section{Menchov-Rademacher Theorem for continuous orthogonal systems}

We start by giving the proof to Theorem \ref{main0}. It is a direct adaptation of the proof of Menchov-Rademacher Theorem  in \cite{ks} but we present it here for the reader's convenience.
\begin{proof}[\bf Proof of Theorem \ref{main0}]
For $j\in\N$, let $P_j(k)=\int_{2^{j-1}}^{2^j}f(r)P(r,k)dr$ and $$S'_j(k)=\sum_{l=1}^jP_l(k)=\int_1^{2^j}f(r)P(r,k)dr\,.$$
Now, $$\|P_j\|_{L^2_\sigma(\R)}^2=\int_\R\left|\int_{2^{j-1}}^{2^j}f(r)P(r,k)dr\right|^2d\sigma(k)=\int_{2^{j-1}}^{2^j}|f(r)|^2dr$$ and so \begin{equation}\sum_{j\in\N}j^2\|P_j\|_{L^2_\sigma(\R)}^2\sim \int_1^\infty |f(r)|^2\log^2(2+r)dr\,.\label{sd_l8}
\end{equation} For any $a>0$, we have 
\[
\sum_{j\in \N}\int_{-a}^a|P_j(k)|{d\sigma(k)}\leq\sum_{j\in\N}\left(\int_{-a}^a|P_j(k)|^2d\sigma(k)\right)^{1/2}\left(\int_{-a}^a{d\sigma(k)}\right)^{1/2}
\le
\]
\[ \sqrt{\sigma([-a,a])} \sum_{j\in\N}\|P_j\|_{L^2_\sigma(\R)}jj^{-1}\leq 
\sqrt{\sigma([-a,a])}\left(\sum_{j\in\N}j^2\|P_j\|_{L^2_\sigma(\R)}^2\right)^{1/2}\left(\sum_{j\in\N}j^{-2}\right)^{1/2}\label{sd_l22}
\]
\[
\lesssim  \sqrt{\sigma([-a,a])}\left(\int_{\R^+}|f(r)|^2\log^2(2+r)dr\right)^{1/2}=\sqrt{\sigma([-a,a])}L^{1/2}\,.
\]
Since $a$ is arbitrary large, by the theorem of Beppo Levi, $\sum_{j\in\N}|P_j(k)|$ converges for $\sigma$-a.e. $k$, as does $\{S'_j(k)\}$.

Let $S'(k)\df \sup_{j\in\N}|S'_j(k)|$ be the maximal function over dyadic partial sums. Since $S'(k)\leq\sum_{j\in\N}|P_j(k)|$,  we have
\begin{align}\label{sd_l11}
\|S'\|_{L^2_\sigma(\R)}&\leq\left\|\sum_{j\in\N}|P_j|\right\|_{L^2_\sigma(\R)}
\leq\sum_{j\in\N}\|P_j\|_{L^2_\sigma(\R)}=\sum_{j\in\N}j^{-1}j\|P_j\|_{L^2_\sigma(\R)}
\lesssim L^{1/2}\,
\end{align}
after applying Cauchy-Schwarz inequality and \eqref{sd_l8}.

For $n\in\{0,1,2,\ldots,2^N\}$, we can write $n=\sum_{m=0}^N\epsilon_m(n)2^{N-m}$ with $\epsilon_m(n)\in\{0,1\}$. For $j\in\{0,1,\ldots,N\}$, let $n_j=\sum_{m=0}^j\epsilon_m(n)2^{N-m}$.

Noting that $\left|\sum_{j=1}^Nx_j\right|^2\le N\sum_{j=1}^N|x_j|^2$, we have:
\begin{align*}
\left|\int_{2^N}^{2^N+n}f(r)P(r,k)dr\right|^2=\left|\sum_{j=1}^N\int_{2^N+n_{j-1}}^{2^N+n_j}f(r)P(r,k)dr\right|^2\le \\ N\sum_{j=1}^N\left|\int_{2^N+n_{j-1}}^{2^N+n_j}f(r)P(r,k)dr\right|^2
\leq N\sum_{j=1}^N\sum_{p=0}^{2^j-1}\left|\int_{2^N+p2^{N-j}}^{2^N+(p+1)2^{N-j}}f(r)P(r,k)dr\right|^2
\end{align*}
and the last expression does not depend on $n$.
Let \[S_j''(k)\df \sup_{0\le n\leq 2^{j}}\left|\int_{2^j}^{2^j+n}f(r)P(r,k)dr\right|\]
denote the maximal function over dyadic interval $[2^j, 2^{j+1}]$.
We apply the above estimate to get
\begin{align}\label{sd_l12}
\|S_N''\|^2_{L^2_\sigma(\R)}&=\int_{\R}\sup_{0\leq n\leq 2^{N}}\left|\int_{2^N}^{2^N+n}f(r)P(r,k)dr\right|^2d\sigma(k)\\\nonumber
&\le\int_{\R}N\sum_{j=1}^N\sum_{p=0}^{2^j-1}\left|\int_{2^N+p2^{N-j}}^{2^N+(p+1)2^{N-j}}f(r)P(r,k)dr\right|^2d\sigma(k)\\ \nonumber
&= N\sum_{j=1}^N\sum_{p=0}^{2^j-1}\int_{\R}\left|\int_{2^N+p2^{N-j}}^{2^N+(p+1)2^{N-j}}f(r)P(r,k)dr\right|^2d\sigma(k)\\ \nonumber
&=N\sum_{j=1}^N\sum_{p=0}^{2^j-1}\int_{2^N+p2^{N-j}}^{2^N+(p+1)2^{N-j}}|f(r)|^2dr=N^2\int_{2^N}^{2^{N+1}}|f(r)|^2dr\,.
\end{align}
Taking $S''=\sup_{j\in\N}S_j''$, we note that $S''\le \left(\sum_{j\in \N}|S_j''|^2\right)^{1/2}$ so
\[
\|S''\|_{L^2_\sigma(\R)}\lesssim \left(\sum_{j\in \N}j^2\int_{2^j}^{2^{j+1}}|f(r)|^2dr   \right)^{1/2}\lesssim L^{1/2}.
\]
Finally, we have
\begin{align*}
\|M\|^2_{L^2_\sigma(\R)}&\lesssim  \int_0^1 |f(r)|^2dr+
\int_\R\sup_{j\in\N}\left|\int_1^{2^j}f(r)P(r,k)dr\right|^2d\sigma(k)+&\\\int_\R\sup_{j\in\N}\sup_{2^j\leq n\leq 2^{j+1}}\left|\int_{2^j}^nf(r)P(r,k)dr\right|^2d\sigma(k)
&=\int_0^1 |f(r)|^2dr+\|S'\|^2+\|S''\|^2\lesssim L\,.
\end{align*}
Convergence of the sequence $\left\{\displaystyle \int_{0}^nf(r)P(r,k)dr\right\}$ for $\sigma$-a.e. $k$ follows from convergence of $\{S_j'(k)\}$ established above  and the estimate
$\displaystyle 
\int_{\R}\sum_{j\in \N}|S_j''|^2d\sigma\lesssim L
$
which yields convergence of $\sum_{j\in \N}|S_j''|^2$ for $\sigma$-a.e. $k$.

\end{proof}

\begin{proof}[\bf Proof of Theorem \ref{main1}] We have
\[
\int_\R\sup_{t\in\R^+}\left|\int_0^tf(r)P(r,k)dr\right|^2\frac{d\sigma(k)}{\kappa(k)}=\int_\R\sup_{t\in\R^+}\left|\int_0^{[t]}f(r)P(r,k)dr+\int_{[t]}^tf(r)P(r,k)dr\right|^2\frac{d\sigma(k)}{\kappa(k)}
\]
\[
\lesssim \|\kappa^{-1}\|_{L^\infty(\R)}\int_\R\sup_{n\in\N}\left|\int_0^nf(r)P(r,k)dr\right|^2d\sigma(k)+\int_\R\sup_{t\in\R^+}\left|\int_{[t]}^tf(r)P(r,k)dr\right|^2\frac{d\sigma(k)}{\kappa(k)}\,.
\]
The first integral was controlled in Theorem \ref{main0}. The second one can be estimated as follows
\begin{equation}\label{smallscale}
\int_\R \sup_{t\in\R^+} \left|\int_{[t]}^t f(r)P(r,k)dr \right|^2 \frac{d\sigma(k)}{\kappa(k)} \leq 
\int_\R \sup_{n\in \mathbb{Z}^+} \left(\int_{n}^{n+1} |f(r)P(r,k)|dr \right)^2 \frac{d\sigma(k)}{\kappa(k)}\le
\end{equation}
\[
 \int_\R \sup_{n\in \mathbb{Z}^+} \left(\left(\int_{n}^{n+1}|f|^2dr\right)\left(\int_n^{n+1} |P(r,k)|^2dr \right) \right)\frac{d\sigma(k)}{\kappa(k)}\le 
 \]
 \[\int_\R \sum_{n=0}^\infty \left(\left(\int_{n}^{n+1}|f|^2dr\right)\left(\int_n^{n+1} |P(r,k)|^2dr \right) \right)\frac{d\sigma(k)}{\kappa(k)}\le \]
 \[
 \sum_{n=0}^\infty \left(\int_{n}^{n+1}|f|^2dr\right)\left(\int_n^{n+1} \left(\int_{\R} \frac{|P(r,k)|^2}{\kappa(k)}d\sigma(k)\right)dr \right) \stackrel{\eqref{sd_d1}}{\le} K\|f\|_2^2\,,
\]
which proves \eqref{sd_l2}.

To establish \eqref{sd_l3}, we notice that 
\[
\int_0^r f(\rho)P(\rho,k)d\rho=\int_0^{[r]} f(\rho)P(\rho,k)d\rho+\int_{[r]}^{r} f(\rho)P(\rho,k)d\rho\,.
\]
The first term has a limit as $r\to\infty$ for $\sigma$-a.e. $k$ as follows from Theorem \ref{main0}. For the second one, we can write
\[
\quad \left|\int_{[r]}^{r} f(\rho)P(\rho,k)d\rho\right|\le \int_{[r]}^{[r]+1} |f(\rho)P(\rho,k)|d\rho
\]
and the last expression goes to $0$ for $\sigma$-a.e. $k$ since 
the series 
\[
\sum_{n\in \N}\left(\int_n^{n+1}|f(r)P(r,k)|dr\right)^2
\]
converges for $\sigma$-a.e. $k$. This convergence follows from the following bound
\begin{eqnarray*}
\int_{\R} \sum_{n\in \N}\left(\int_n^{n+1}|f(r)P(r,k)|dr\right)^2\frac{d\sigma}{\kappa}\le \int_{\R} \sum_{n\in \N}\left(\left(\int_n^{n+1}|f(r)|^2dr\right)\left( \int_{n}^{n+1}|P(r,k)|^2dr\right)\right)\frac{d\sigma}{\kappa}\le\\
\left(\sup_{r\ge 0} \int_{\R} \frac{|P(r,k)|^2}{\kappa}d\sigma\right) \sum_{n\in \N}\int_n^{n+1}|f(r)|^2dr\stackrel{\eqref{sd_d1}}{<}\infty\,.
\end{eqnarray*}

\end{proof}

Before giving the proof of the Lemma \ref{l1}, we list some basic properties of Krein systems which will be needed later in the text. We start by making a remark that 
\begin{equation}\label{sd_l33}
P(r,k)=e^{irk}\overline{P_*(r,k)},
\end{equation}
provided that $k\in \R$. This identity follows directly from \eqref{sa} and can be found in, e.g., \cite{d1}.

Next, we consider an important case when $A\in L^2(\mathbb{R}^+)$. In \cite{den1} (see also original Krein's paper \cite{krein}), it was shown that the following properties hold under this condition:

$\bullet$ There is a function $\Pi(k), k\in \mathbb{C}^+$ such that
\begin{equation}\label{li}
\lim_{r\to\infty}P_*(r,k)=\Pi(k)
\end{equation}
uniformly over compact sets in $\mathbb{C}^+$. This $\Pi$ is outer and the orthogonality measure $\sigma$ can be written as follows
\begin{equation}\label{sd_l1}
d\sigma=\frac{dk}{2\pi|\Pi(k)|^2}+d\sigma_s,
\end{equation}
where $\sigma_s$ is its singular part.

$\bullet$ Integrating the second equation in \eqref{sa}, we have
\begin{equation}\label{sm1}
P_*(r,k)=1-\int_0^r A(\rho)P(\rho,k)d\rho\,.
\end{equation}
Therefore
\[
1-P_*(r,k)=\int_0^r A(\rho)P(\rho,k)d\rho\to \widetilde A(k)\df \int_0^\infty A(\rho)P(\rho,k)d\rho\,,
\]
when $r\to\infty$
and convergence is in $L^2(\mathbb{R},\sigma)$ norm. On the other hand,  the formula (12.37) in \cite{den1} gives
\[
\widetilde A(k)=1-\Pi(k)\cdot \chi_{E_s^c},
\]
where $E_s^c$ denotes the complement to $E_s$, the support of $\sigma_s$. Therefore, 
 \begin{equation}\label{sd_uu}
 \lim_{r\to\infty}\|P_*(r,k)-\Pi(k)\cdot \chi_{E_s^c}\|_{2,\sigma}=0\,.
 \end{equation}

$\bullet$ From $\eqref{sm1}$ and orthogonality, we get
\[
\int_{\mathbb{R}}|P_*(r,k)-1|^2d\sigma=\int_0^r |A(\rho)|^2d\rho\,.
\]

\begin{proof}[\bf Proof of Lemma \ref{l1}] 

The second equation in \eqref{sa} gives
\begin{equation}\label{duid}
P_*(r_2,k)-P_*(r_1,k)=-\int_{r_1}^{r_2}A(r)P(r,k)dr\,.
\end{equation}
Theorem \ref{main1} yields necessary estimate on the maximal function and convergence of $P_*(r,k)$ $\sigma$-a.e. The limit is equal to $\Pi$ from \eqref{li} due to \eqref{sd_uu}.
\end{proof}

\bigskip
\section{Wave operators for Schr\"odinger evolution: proof of Theorem \ref{t2}}

We start this section by describing a connection between Krein systems and Dirac and Schr\"odinger operators on $\R^+$.
Consider the Krein system with  coefficient $A\in L^2_{\rm loc}(\mathbb{R}^+)$. It corresponds to Dirac operator 
\begin{equation}\label{sd_1}
\mathcal{D}=\left(
\begin{array}{cc}
-b & \partial_x-a\\
-\partial_x-a &b
\end{array}
\right)
\end{equation}
defined on Hilbert space $(f_1,f_2)\in L^2(\mathbb{R}^+)\times L^2(\mathbb{R}^+)$, where $a(x)=2\Re A(2x), b(x)=2\Im A(2x)$ with the boundary condition $f_2(0)=0$. Indeed, define real-valued functions $\phi$ and $\psi$ by writing $\phi(x,k)+i\psi(x,k)\df P(2x,k)e^{-ikx}$. It can be checked \cite{d1,krein} that $(\phi,\psi)$ are generalized eigenfunctions for Dirac operator  \eqref{sd_1} and that $2\sigma$ is its spectral measure.  Define $\{\mathcal E(x,k)\},x\ge 0$ by 
\begin{equation}
\mathcal{E}(x,k)\df P(2x,k)e^{-ixk}.\label{sd_l77}
\end{equation} It turns out that this is also  continuous orthonormal system with respect to $\sigma$, i.e.,
\begin{equation}\label{df2}
\int_{\mathbb{R}} \left|  \int_0^\infty f(x)\mathcal{E}(x,k)dx  \right|^2d\sigma=\|f\|_2^2\,, 
\end{equation}
for every $f\in L^2(\mathbb{R}^+)$ (see \cite{den1, krein}). 
Making an extra assumption that $A$ is real-valued, i.e., that  $b=0$, and absolutely continuous on $\R^+$ and taking the square of $\mathcal{D}$ reveals the connections between Dirac and Schr\"odinger operators. Indeed, 
\begin{equation}\label{ds}
\mathcal{D}^2=\left(
\begin{array}{cc}
H_1 & 0\\
0 & H_2
\end{array}
\right)
\end{equation}
where $H_1f=-\partial^2_{xx}f+q_1f, f'(0)+a(0)f(0)=0$, $H_2f=-\partial^2_{xx}f+q_2f, f(0)=0$, 
\[
q_1=a^2-a', q_2=a^2+a'\,.
\]
Later in the proof, we will use the spectral decomposition for Dirac $\cal{D}$ and the formula \eqref{ds} to write a suitable  expression for $e^{itH_{2}}$.

The following result implies Theorem \ref{t2} thanks to Lemma \ref{l1}.
\begin{theorem}\label{sd_t1} Suppose the coefficient $A$ in the Krein system is real and absolutely continuous, $A\in L^2(\mathbb{R}^+), A'\in L^\infty(\mathbb{R}^+)$, and 
\begin{equation}\label{extra}
\lim_{\rho\to\infty}\int_{\mathbb{R}} \left( \sup_{\rho<r_1<r_2} \left|\int_{r_1}^{r_2}A(r)P(r,k)dx \right|\right)^2\frac{d\sigma}{1+k^2}=0\,.
\end{equation}
Let $a(x)=2A(2x)$ and let $q$ be real-valued function on $\mathbb{R}^+$ satisfying $q\in L^1(\mathbb{R}^+).$ Then, taking two operators $H=-\partial_{xx}^2+a'+q$ and $H_0=-\partial_{xx}^2$ both with Dirichlet boundary condition at zero, we get existence of wave operators $W^{\pm }(H,H_0)$.
\end{theorem}
This Theorem is the central technical result of our paper. Before giving its proof, we state the following Lemma.
\begin{lemma}\label{prince}
Suppose $t\ge 0$, $\mu$ is a measure on $\mathbb{R}$, and $p(k),p_t(k)\in L_\mu^2(\mathbb{R})$. Let $\|p\|_{2,\mu}=1$ and 
\begin{equation}\label{sd_l45}
\lim_{t\to\infty}\|p_t\|_{2,\mu}=1, \quad
\lim_{t\to\infty}\int_{\Delta}|p-p_t|^2d\mu=0
\end{equation}
for every interval $\Delta\subset \mathbb{R}$. Then, $\lim\limits_{t\to\infty}\|p-p_t\|_{2,\mu}=0$.
\end{lemma}
\begin{proof}
The proof is based on a standard exhaustion principle. For every $\epsilon\in (0,1)$, we can choose $L>0$ such that $\int_{\Delta^c}|p|^2d\mu \le\epsilon$ where $\Delta\df [-L,L]$. By \eqref{sd_l45}, there is $T$ so that 
\[
|1-\|p_t\|^2_{2,\mu}|<\epsilon, \quad
\int_{\Delta}|p-p_t|^2d\mu<\epsilon\] for $t>T$. Thus, for $t>T$, we also have
\begin{eqnarray*}
\int_{\Delta^c}|p_t|^2d\mu=\|p_t\|_{2,\mu}^2-\int_{\Delta}|p_t|^2d\mu=\\
\|p_t\|_{2,\mu}^2-\left(1-\int_{\Delta^c}|p|^2d\mu-\int_{\Delta}(|p|^2-|p_t|^2)d\mu\right)\le\\
|\|p_t\|_{2,\mu}^2-1|+\int_{\Delta^c}|p|^2d\mu+\left|\int_{\Delta}(|p|^2-|p_t|^2)d\mu\right|\lesssim\\
 \epsilon+\sqrt\epsilon,
\end{eqnarray*}
where we used triangle inequality to estimate
\begin{eqnarray*}
\left|\int_{\Delta}(|p|^2-|p_t|^2)d\mu\right|=\Bigl|\|p\|^2_{L^2_{\mu}(\Delta)}-\|p_t\|_{L^2_{\mu}(\Delta)}^2\Bigr|=\\
\Bigl( \|p\|_{L^2_{\mu}(\Delta)}+\|p_t\|_{L^2_{\mu}(\Delta)}\Bigr)\cdot
\Bigl|\|p\|_{L^2_{\mu}(\Delta)}-\|p_t\|_{L^2_{\mu}(\Delta)}\Bigr|\lesssim
 \|p-p_t\|_{L^2_{\mu}(\Delta)} \le \sqrt \epsilon\,.
\end{eqnarray*}
Thus, \[\int_{\mathbb{R}}|p-p_t|^2d\mu=
\int_{\Delta}|p-p_t|^2d\mu+\int_{\Delta^c}|p-p_t|^2d\mu \le \epsilon+2\int_{\Delta^c}|p|^2d\mu+2\int_{\Delta^c}|p_t|^2d\mu \lesssim \sqrt\epsilon
\]
 for $t>T$ and the proof is finished.
\end{proof}

\begin{proof}[\bf Proof of Theorem \ref{sd_t1}]
Since $a^2,q\in L^1(\mathbb{R}^+)$ and relative trace class perturbations do not change existence of wave operators (Birman-Kuroda Theorem, \cite{rs33}, p. 27), it is enough to consider $H=H_2=a'+a^2$. 
Take $f\in L^2(\mathbb{R}^+)$. We need to prove existence of
\begin{equation}\label{sd_l66}
\lim_{t\to\pm \infty} e^{itH}e^{-itH_0}f\,,
\end{equation}
where the limit is understood in $L^2(\mathbb{R}^+)$ topology.
Notice that, since both groups $e^{itH}$ and $e^{-itH_0}$ preserve $L^2(\R^+)$ norm, it is enough to prove existence of the limit for every $f\in \cal{T}$ where $\cal{T}$ is any dense subset in $L^2(\mathbb{R}^+)$. We define $\cal{T}$ as follows: $\cal{T}\df \{f: \widehat{f_o}\in C_c^\infty(\mathbb{R}), 0\notin \supp \widehat{f_o}\}$, where $f_o$ denotes the odd extension of $f$ to $\mathbb{R}$. From now on, we assume that $f\in \cal{T}, \|f\|_2=1$ and that $t\to +\infty$ in \eqref{sd_l66} (the case $t\to -\infty$ can be handled similarly). Denote $f_+\df (\widehat{f_o}\cdot \chi_{\xi>0})^{\vee}$, $f_-\df (\widehat{f_o}\cdot \chi_{\xi<0})^{\vee}$. Working on the Fourier side, we get
\[
e^{-itH_0}f=\frac{1}{\pi}
\int_{\mathbb{R}}e^{-it\xi^2}\left(\int_{\mathbb{R}^+} f(u)\sin(\xi u)du \right)\sin(\xi x)d\xi=\frac{1}{2\pi}\int_{\mathbb{R}}e^{-it\xi^2}\left(\int_{\mathbb{R}} f_o(u)e^{-i \xi u}du \right)e^{i\xi x}d\xi\,.
\]
The last expression is equal to the restriction of $e^{it\partial^2_{xx}}f_o$ to $\mathbb{R}^+$, where $\partial^2_{xx}$ is considered on all of $\mathbb{R}$. The large time asymptotics of $e^{it\partial^2_{xx}}h$ for $h\in L^2(\R)$  is known and given in Lemma \ref{a6} from Appendix. Since $\widehat f_o(\xi)=\widehat f_+(\xi)$ for $\xi>0$,  it is  enough to show that
\begin{equation}\label{sd_5}
I\df \frac{e^{itk^2}}{1+i}   \int_0^\infty \frac{e^{ix^2/(4t)}}{\sqrt t}\widehat f_+(x/(2t))   \psi(x,k)dx
\end{equation}
has a limit in $L^2(\mathbb{R},2\sigma)$ when $t\to +\infty$. Indeed, the spectral measure for Dirac operator $\cal{D}$ is equal to $2\sigma$, the generalized eigenfunctions are $(\phi,\psi)$, and the Schr\"odinger operator is related to Dirac by \eqref{ds} so we can use spectral decomposition for Dirac operator to compute $e^{itH}$ where $H=H_2$. To this end, we will use the following generalized Fourier transform  
\[
\left(\begin{matrix}f_1\\f_2
\end{matrix}\right)\rightarrow \cal{F}=\int_0^\infty f_1(x)\phi (x,k)dx+\int_0^\infty f_2(x)\psi(x,k) dx
\]
and the analog of Plancherel's Theorem
\[
 \|f_1\|_2^2+\|f_2\|_2^2=\|{\cal F}\|_{2,2\sigma}^2\,.
\]
Since $f\in \cal{T}$,  $\widehat f_+$ is supported on some interval $[a,b]$ and $a>0$. Use \eqref{sd_l33} and substitute
\[
\psi(x,k)=\frac{\overline{P_*}(2x,k)e^{ikx}-P_*(2x,k)e^{-ikx}}{2i}
\]
into \eqref{sd_5} to get
\begin{equation}\label{sort}
I=I_1-I_2\,,
\end{equation}
where
\begin{eqnarray*}
I_1=\frac{e^{itk^2}}{2i(1+i)}\int_{2at}^{2bt}\frac{e^{ix^2/(4t)}}{\sqrt t}\widehat f_+(x/(2t))\overline{P_*}(2x,k)e^{ikx}dx, \\
I_2=\frac{e^{itk^2}}{2i(1+i)}\int_{2at}^{2bt}\frac{e^{ix^2/(4t)}}{\sqrt t}\widehat f_+(x/(2t)){P_*}(2x,k)e^{-ikx}dx\,.
\end{eqnarray*}
Consider $I_2$, the analysis of $I_1$ is similar. Integrating by parts, we get
\[
\int_{2at}^{2bt} {P_*}(2x,k) \left(\int_{2at}^x \frac{e^{iu^2/(4t)}}{\sqrt t}\widehat f_+(u/(2t))e^{-iku}du\right)'dx=
\]
\[={P_*}(4bt,k)\int_{2at}^{2bt}\frac{e^{iu^2/(4t)}}{\sqrt t}\widehat f_+(u/(2t))e^{-iku}du-J_2\,,
\]
where, thanks to the second equation in \eqref{sa},
\[
J_2=\int_{2at}^{2bt} 2A(2x)P(2x,k) \left(\int_{2at}^x \frac{e^{iu^2/(4t)}}{\sqrt t}\widehat f_+(u/(2t))e^{-iku}du\right)dx\,.
\]
For the first term, we can write
\[
{P_*}(4bt,k)\int_{2at}^{2bt}\frac{e^{iu^2/(4t)}}{\sqrt t}\widehat f_+(u/(2t))e^{-iku}du=
\]
\[
({P_*}(4bt,k)-\Pi(k)\cdot \chi_{E_s^c})\int_{2at}^{2bt}\frac{e^{iu^2/(4t)}}{\sqrt t}\widehat f_+(u/(2t))e^{-iku}du
\]
\[
+\Pi(k)\cdot \chi_{E_s^c}\int_{2at}^{2bt}\frac{e^{iu^2/(4t)}}{\sqrt t}\widehat f_+(u/(2t))e^{-iku}du\,.
\]
From \eqref{ff1}, we get
\[
\sup_{t>1}\left\|\int_{2at}^{2bt}\frac{e^{iu^2/(4t)}}{\sqrt t}\widehat f_+(u/(2t))e^{-iku}du\right\|_{L^\infty(\R)}<C_{(f)}
\]
and \eqref{sd_uu} implies 
\[
\lim_{t\to+\infty}\left\|({P_*}(4bt,k)-\Pi(k)\cdot \chi_{E_s^c})\int_{2at}^{2bt}\frac{e^{iu^2/(4t)}}{\sqrt t}\widehat f_+(u/(2t))e^{-iku}du\right\|_{2,\sigma}= 0\,.
\]
From \eqref{sd_l1} and \eqref{sd_7}, we obtain
\begin{equation}\label{den1}
\lim_{t\to\infty}\left\|\frac{e^{itk^2}\Pi(k)}{2i(1+i)}\cdot \chi_{E_s^c}\cdot \int_{2at}^{2bt}\frac{e^{iu^2/(4t)}}{\sqrt t}\widehat f_+(u/(2t))e^{-iku}du- \frac{\sqrt{2\pi}\Pi(k)}{2i}\cdot \chi_{E_s^c} \widehat f_+(k)\right\|_{2,\sigma}=0\,.
\end{equation}
The analysis for $I_1$ is analogous, it also gives the main term converging to 
\[
\frac{\sqrt{2\pi} \cdot \overline{\Pi(k)}}{2i}\cdot  \chi_{E_s^c} \widehat f_+(-k)
\] and a correction which we call $J_1$.
Consider $J_1$ and $J_2$. We claim that if we show that 
\begin{equation}\label{sd_l55}
\lim_{t\to\infty}\int_{\Delta}|J_1|^2d\sigma=0, \quad \lim_{t\to\infty}\int_{\Delta}|J_2|^2d\sigma=0
\end{equation}
for every interval $\Delta\subset \mathbb{R}$, then  the proof of  Theorem \ref{sd_t1} will be finished after application of  Lemma~\ref{prince}. Indeed, in this lemma, we set $\mu=2\sigma$, $p_t=I$ and the limiting function $p$ is
\[
p=\chi_{E_s^c}\cdot \sqrt{2\pi}\frac{\overline{\Pi(k)}\widehat f_+(-k)-\Pi(k)\widehat f_+(k)}{2i}\,.
\]
To apply Lemma \ref{prince}, we notice that $\|I\|_{2,2\sigma}\to 1$ by Lemma \ref{a6}. Moreover, \eqref{sd_l1} gives $\|p\|_{2,2\sigma}=\|f\|_2=1$.

We will prove the second identity in \eqref{sd_l55}, the first one can be obtained similarly. For $J_2$, we have 
\[
J_2=-2\int_{2at}^{2bt} A(2x)P(2x,k)\left(\int_{2at}^x \frac{e^{i(u^2/(4t)-ku)}}{\sqrt t}\widehat f_+(u/(2t))du\right) dx\,.
\]
One can write
\[
\int_{2at}^x \frac{e^{i(u^2/(4t)-ku)}}{\sqrt t}\widehat f_+(u/(2t))du=\int_{2at}^0 \frac{e^{i(u^2/(4t)-ku)}}{\sqrt t}\widehat f_+(u/(2t))du+\int_{0}^x \frac{e^{i(u^2/(4t)-ku)}}{\sqrt t}\widehat f_+(u/(2t))du\,.
\]
The first term does not depend on $x$ and we can use \eqref{ff1} and \eqref{ort1} to write
\begin{equation}\label{lkl}
\left\|\int_{2at}^{2bt} A(2x)P(2x,k)\left(\int_{2at}^0 \frac{e^{i(u^2/(4t)-ku)}}{\sqrt t}\widehat f_+(u/(2t))du \right)dx\right\|_{2,\sigma}\le C_{(f)}\int_{at}^{bt}|A(x)|^2dx\,,
\end{equation}
where the last expression converges to zero as $t\to\infty$. For the other term, we have
\[
\int_{0}^x \frac{e^{i(u^2/(4t)-ku)}}{\sqrt t}\widehat f_+(u/(2t))du=
e^{-itk^2}\int_{0}^x \frac{e^{i(u/(2\sqrt t)-k\sqrt t)^2}}{\sqrt t}\widehat f_+(u/(2t))du\,.
\]
The integral can be rewritten as
\[
\int_{0}^x \frac{e^{i(u/(2\sqrt t)-k\sqrt t)^2}}{\sqrt t}\widehat f_+(u/(2t))du=
\]
\[\int_{-\infty}^x \frac{e^{i(u/(2\sqrt t)-k\sqrt t)^2}}{\sqrt t}\widehat f_+(u/(2t))du-\int_{-\infty}^0 \frac{e^{i(u/(2\sqrt t)-k\sqrt t)^2}}{\sqrt t}\widehat f_+(u/(2t))du\,.
\]
The second term is $x$--independent so its contribution is negligible by the argument identical to \eqref{lkl}. For the first one, we change variables and write, using the same variable $u$,
\begin{eqnarray}\label{d2}
\int_{-\infty}^x \frac{e^{i(u/(2\sqrt t)-k\sqrt t)^2}}{\sqrt t}\widehat f_+(u/(2t))du=2\int_{-\infty}^{(x-2kt)/{2\sqrt t}}  e^{iu^2}\widehat f_+(k+u/\sqrt t)du\\
=2\int_{-\infty}^{(x-2kt)/{2\sqrt t}}  e^{iu^2}\Bigl(\widehat f_+(k+u/\sqrt t)-\widehat f_+(k)\Bigr)du+2\widehat f_+(k)\int_{-\infty}^{(x-2kt)/{2\sqrt t}}  e^{iu^2}du\nonumber\,.
\end{eqnarray}
We can continue as follows
\begin{eqnarray*}
\int_{-\infty}^{(x-2kt)/{2\sqrt t}}  e^{iu^2}\Bigl(\widehat f_+(k+u/\sqrt t)-\widehat f_+(k)\Bigr)du=\int_{-\infty}^{0}  e^{iu^2}\Bigl(\widehat f_+(k+u/\sqrt t)-\widehat f_+(k)\Bigr)du+\\
\int_{0}^{(x-2kt)/{2\sqrt t}}  e^{iu^2}\Bigl(\widehat f_+(k+u/\sqrt t)-\widehat f_+(k)\Bigr)du\,.
\end{eqnarray*}
The first term in the right-hand side does not depend on $x$ and it is uniformly bounded in $k\in \mathbb{R}$ and $t\ge 1$ as can be seen by integrating by parts. Thus, its contribution to $\|J_2\|_{L^2_\sigma(\Delta)}$ is also negligible. 

We want to apply Lemma~\ref{a1} from Appendix to the second term. Since we are  interested in $k\in \Delta$ and $x\in [at,bt]$, then $|(x-2kt)/2t|<C_{(a,b,\Delta)}$. Hence, the Lemma is applicable with $\epsilon=1/\sqrt t, g(u)=\widehat f_+(k+u)-\widehat f_+(k)$ which gives
\[
\left|\int_{0}^{(x-2kt)/{2\sqrt t}}  e^{iu^2}\Bigl(\widehat f_+(k+u/\sqrt t)-\widehat f_+(k)\Bigr)du\right|\le C_{a,b,\Delta,f}/\sqrt t\,.
\]
The proof of Lemma \ref{a1} shows that this bound is uniform in $k\in \Delta$. We substitute it and apply \eqref{sd_l21} along with generalized Minkowski inequality to get
\begin{eqnarray*}
\left(\int_{\Delta} \left|    \frac{1}{\sqrt t}\int_{2at}^{2bt} |A(2x)P(2x,k)|\right|^2 d\sigma \right)^{1/2}\lesssim 
\\
\frac{1}{\sqrt t}\int_{2at}^{2bt} |A(2x)|\cdot \left(\int_\Delta |P(2x,k)|^2d\sigma\right)^{1/2}dx\stackrel{\eqref{sd_l21}}{\lesssim }
\\
\frac{C_{(\Delta,\|A\|_{\rm St})}}{\sqrt t}\int_{2at}^{2bt} |A(2x)|dx\le C_{(\Delta,a,b,\|A\|_{\rm St})}\left(   \int_{at}^{bt}|A(x)|^2dx\right)^{1/2}\,
\end{eqnarray*}
and the last expression converges to zero
when $t\to +\infty$.
We are only left with controlling the contribution from the last term in \eqref{d2}, i.e.,
\[
\widehat f_+(k)\int_{2at}^{2bt} A(2x)P(2x,k) \left(\int_{0}^{(x-2k t)/(2\sqrt t)}e^{iu^2}du\right) dx\,.
\]
Let us write partition of unity \begin{equation}
1=\mu_{-}+\mu_0+\mu_{+},\label{l3}
\end{equation} where $\mu_0$ is even, smooth, supported in $(-2,2)$ and 
\[
0\le \mu_0\le 1, \quad \mu_0=1\,\, {\rm if}\,\, |x|<1\,.
\] Function $\mu_+$ is supported on $(1,\infty)$ and is non-decreasing, $\mu_-(x)\df \mu_+(-x)$. Then, 
\[
\int_{0}^{(x-2k t)/(2\sqrt t)}e^{iu^2}du=\left(\int_{0}^{(x-2k t)/(2\sqrt t)}e^{iu^2}du\right)\Bigl(\mu_-((x-2k t)/(2\sqrt t))+\mu_0(\cdot)+\mu_+(\cdot)\Bigr)\,.
\]
We will apply the following trick several times. Notice that the function $F(x)\df (\int_0^{x} e^{iu^2}du)\mu_0(x)\in C_c^\infty(\mathbb{R})$  thus $\widehat F\in L^1(\mathbb{R})$ and we can write
\[
F((x-2k t)/(2\sqrt t))=\frac{1}{\sqrt{2\pi}}\int_{\mathbb{R}}\widehat F(\xi)\exp(i\xi (x-2k t)/(2\sqrt t))d\xi\,.
\]
Then,
\begin{eqnarray*}
\widehat f_+(k)\int_{2at}^{2bt} A(2x)P(2x,k) \left(\mu_0((x-2kt)/(2\sqrt t))\int_{0}^{(x-2k t)/(2\sqrt t)}e^{iu^2}du\right) dx=
\\
\widehat f_+(k)\int_{2at}^{2bt} A(2x)P(2x,k) F((x-2kt)/(2\sqrt t)) dx
=\\
\frac{1}{\sqrt{2\pi}}\int_{\mathbb{R}}\widehat F(\xi)\left(\widehat f_+(k)e^{-i\xi k\sqrt t}    \int_{2at}^{2bt} A(2x)P(2x,k)\exp(i\xi x/(2\sqrt t))dx\right)d\xi\,.
\end{eqnarray*}
We use generalized Minkowski inequality  and \eqref{ort1} to estimate the last quantity as follows
\begin{eqnarray*}
\left\|\frac{1}{\sqrt{2\pi}}\int_{\mathbb{R}}\widehat F(\xi)\left(\widehat f_+(k)e^{-i\xi k\sqrt t}    \int_{2at}^{2bt} A(2x)P(2x,k)\exp(i\xi x/(2\sqrt t))dx\right)d\xi\right\|_{2,\sigma}
\lesssim \\
\left(\int_{\mathbb{R}} |\widehat F(\xi)|d\xi\right) \|\widehat f_+\|_{\infty}\left(\int_{at}^{bt}|A(x)|^2dx\right)^{1/2}
\end{eqnarray*}
and the last quantity converges to zero when $t\to\infty$.
We apply similar strategy to other terms. 
\begin{eqnarray*}
\left(\int_{0}^{(x-2k t)/(2\sqrt t)}e^{iu^2}du\right)\mu_{+}((x-2k t)/(2\sqrt t))=C\mu_+((x-2k t)/(2\sqrt t))\\-\left(\int_{(x-2k t)/(2\sqrt t)}^\infty e^{iu^2}du\right)\mu_{+}((x-2k t)/(2\sqrt t))\,,
\end{eqnarray*}
where $C\df \int_{0}^{\infty}e^{iu^2}du$.
Consider
\[
\int_{2at}^{2bt} A(2x)P(2x,k) \mu_+((x-2k t)/(2\sqrt t))dx=\int_{2at}^{2bt} \left(\int_{2at}^x A(2u)P(2u,k)du\right)' \mu_+((x-2k t)/(2\sqrt t))dx
\]
\[
=\left(\int_{2at}^{2bt}  A(2u)P(2u,k)du \right)   \mu_+((b-k)\sqrt t)-\int_{2at}^{2bt} \left(\int_{2at}^x A(2u)P(2u,k)du\right)\frac{\mu'_+((x-2k t)/(2\sqrt t))}{2\sqrt t}dx\,.
\]
The first term gives contribution
\[
\int_{\mathbb{R}}\left|  \widehat f_+(k) \left(\int_{2at}^{2bt}  A(2u)P(2u,k)du \right)   \mu_+((b-k)\sqrt t)       \right|^2d\sigma \lesssim  \|\widehat f_+\|_\infty^2 \int_{2at}^{2bt}|A(2u)|^2du\,
\]
and the last quantity converges to zero when $t\to\infty$.
For the second one, we can write an estimate
\begin{eqnarray}\label{key}
\left|\int_{2at}^{2bt} \left(\int_{2at}^x A(2u)P(2u,k)du\right)\frac{\mu'_+((x-2k t)/(2\sqrt t))}{2\sqrt t}dx\right|\le\\ \left(\sup_{2at<r_1<r_2}\left|\int_{r_1}^{r_2}A(2u)P(2u,k)du\right|\right)\cdot \int_{2at}^{2bt}\left|\frac{\mu'_+((x-2k t)/(2\sqrt t))}{2\sqrt t}\right|dx\,.\nonumber
\end{eqnarray}
Since  $\mu_+$ was chosen to be non-decreasing, one obtains
\[
\int_{2at}^{2bt}\left|\frac{\mu'_+((x-2k t)/(2\sqrt t))}{2\sqrt t}\right|dx\lesssim  1\,.
\]
Under the assumptions of the theorem, we get
\[
\left\||\widehat f_+|\cdot \sup_{2at<r_1<r_2}\left|\int_{r_1}^{r_2}A(2u)P(2u,k)\right|\right\|_{L^2_\sigma(\Delta)}\to 0
\]
when $t\to\infty$. Consider the expression
\[
\left(\int_{(x-2k t)/(2\sqrt t)}^\infty e^{iu^2}du\right)\mu_{+}((x-2k t)/(2\sqrt t))
\]
and apply Lemma \ref{a2} from Appendix to write it as
\[
\left(\int_{(x-2k t)/(2\sqrt t)}^\infty e^{iu^2}du\right)\mu_{+}((x-2k t)/(2\sqrt t))=
\]
\[(2\pi)^{-1/2}e^{ix^2/(4t)}e^{-ixk}e^{ik^2t}\int_{\mathbb{R}} e^{i\xi(x-2kt)/(2\sqrt t)}\Psi(\xi)d\xi\,,
\]
where $\Psi\in L^1(\mathbb{R})$. Then,
\begin{eqnarray*}
\int_{2at}^{2bt} A(2x)P(2x,k)e^{ix^2/(4t)}e^{-ixk}e^{ik^2t}\left(\int_{\mathbb{R}} e^{i\xi(x-2kt)/(2\sqrt t)}\Psi(\xi)d\xi\right) dx\\
=e^{ik^2t}\int_\mathbb{R} \Psi(\xi) e^{-i\xi k\sqrt t}\left(\int_{2at}^{2bt}A(2x)e^{ix^2/(4t)}e^{i\xi x/(2\sqrt t)}\cal{E}(x,k)dx\right)d\xi\,,
\end{eqnarray*}
where $\cal{E}(x,k)=P(2x,k)e^{-ikx}$ was introduced in \eqref{sd_l77}. 
Using generalized Minkowski inequality and \eqref{df2}, we get
\begin{eqnarray*}
\left\|\widehat f_+(k)\cdot e^{ik^2t}\int_{\mathbb{R}} \Psi(\xi) e^{-i\xi k\sqrt t}\left(\int_{2at}^{2bt}A(2x)e^{ix^2/(4t)}e^{i\xi x/(2\sqrt t)}\cal{E}(x,k)dx\right)d\xi\right\|_{2,\sigma}\lesssim \\ \|\widehat f_+\|_\infty\cdot \left(\int_{\mathbb{R}}|\Psi(\xi)|d\xi\right)\cdot  \left(\int_{2at}^{2bt}|A(2x)|^2dx\right)^{1/2}
\end{eqnarray*}
and the last quantity converges to zero when $t\to\infty$.

The contribution from the term
\[
\left(\int_{0}^{(x-2k t)/(2\sqrt t)}e^{iu^2}du\right)\mu_{-}((x-2k t)/(2\sqrt t))
\]
can be handled in the same way. Thus, 
\[
\lim_{t\to\infty}\int_{\Delta}|J_2|^2d\sigma= 0
\]
and our Theorem is proved.
\end{proof}
\noindent {\bf Remark.} Notice that we had to use an additional assumption about the maximal function \eqref{extra} only when handling~\eqref{key}. It is an intriguing question whether this extra hypothesis can be dropped.

\section{Appendix}

In this Appendix, we collect results that are used in the main text. Although some of them are standard, we provide their proofs for completeness. \medskip

\noindent {\bf Proof of Lemma \ref{sd_al}.}
In Section 13 of \cite{d1}, the following formula for the Green's function of operator $\cal{D}$ (i.e., the integral kernel of $R_z=(\cal{D}-z)^{-1}$) was obtained
\[
G(x,y,z)=\left(\begin{array}{cc}G_{11}(x,y,z)&G_{12}(x,y,z)\\G_{21}(x,y,z)&G_{22}(x,y,z)\end{array}\right)=
\]
\begin{equation}
\left( \begin{array}{cc}\displaystyle 
\int_{\R} \frac{\phi(x,k)\phi(y,k)}{k-z}d\sigma_d(k) &\displaystyle  \int_{\R} \frac{\phi(x,k)\psi(y,k)}{k-z}d\sigma_d(k) \\\displaystyle 
\int_{\R} \frac{\psi(x,k)\phi(y,k)}{k-z}d\sigma_d(k) & \displaystyle \int_{\R} \frac{\psi(x,k)\psi(y,k)}{k-z}d\sigma_d(k) 
\end{array}\right)\label{sd_ff}
\end{equation}
and $\sigma_d=2\sigma$. We now introduce an auxiliary parameter $\rho\in [1,\infty)$ to be chosen later as $\rho\sim 1+\|A\|_{\rm St}^2$. Since $|P(2x,k)|^2=\phi^2(x,k)+\psi^2(x,k)$ and $\sup_{k\in \R} (k^2+\rho^2)/(k^2+1)\lesssim \rho^2$, then
\begin{equation}\label{agk9}
\sup_{x\ge 0}\int_\R \frac{|P(x,k)|^2}{k^2+1}d\sigma=\sup_{x\ge 0}\int_\R \frac{(k^2+\rho^2)|P(x,k)|^2}{(k^2+\rho^2)(k^2+1)}d\sigma\lesssim \rho \sup_{x\ge 0}\int_\R \frac{\rho|P(x,k)|^2}{k^2+\rho^2}d\sigma\,.
\end{equation}
Hence, we only need to prove that
\begin{equation}\label{agk0}
\sup_{x\ge 0} \Im (G_{11}(x,x,i\rho)+G_{22}(x,x,i\rho))\lesssim 1\,.
\end{equation}
To control $G(x,y,i\rho)$, i.e., the integral kernel of the resolvent $R_{i\rho}$, we will use the  standard perturbation series. If $R_{i\rho}^0$   denotes the resolvent of free Dirac operator,  we write the second resolvent identity:
\[
R_{i\rho}=R^0_{i\rho}-R_{i\rho}VR^0_{i\rho}, \quad V\df \left(
\begin{array}{cc} -b & -a\\-a& b
\end{array}\right)
\]
and iterate it to get the series
\begin{equation}\label{sd_sd}
R_{i\rho}=R^0_{i\rho}-R^0_{i\rho}VR^0_{i\rho}+R^0_{i\rho}VR^0_{i\rho}VR^0_{i\rho}+\ldots\,.
\end{equation}
In the series \eqref{sd_sd}, each term starting with the second one takes the form $(-1)^{j+1} (R^0_{i\rho}V)^j(R^0_{i\rho}VR_{i\rho}^0)$ and $j=0,1,2,\ldots$. If we denote its kernel by $k_j(x,y)$, then
\begin{equation}\label{agk6}
G(x,y,i\rho)=G^0(x,y,i\rho)-k_0(x,y)+k_1(x,y)+\ldots
\end{equation}
and $G^0(x,y,z)$ stands for the Green's function of free Dirac operator. Next, we will show  convergence of this series for suitable choice of parameter $\rho$ and will provide an estimate for it.

First, we claim that for every $j=0,1,\ldots$, we have
\begin{equation}\label{agk1}
\|k_j(x,y)\|\le C^{j+1} \frac{e^{-\rho |x-y|/2}\|A\|_{\rm St}^{j+1}}{\rho^{(j+1)/2}}\,,
\end{equation}
where $C$ is an absolute constant to be specified below.
We will prove \eqref{agk1} by induction. To this end, we use formula \eqref{sd_ff} and residue calculus to obtain the bound \[\|G^0(x,y,i\rho)\|\lesssim e^{-\rho|x-y|}+e^{-\rho(x+y)}\lesssim e^{-\rho|x-y|}\,.\]  
Thus, for $k_0(x,y)$, we have
\[
\|k_0(x,y)\|\lesssim \int_0^\infty e^{-\rho|x-\xi|}|\alpha(\xi)|e^{-\rho|y-\xi|}d\xi, \quad \alpha\df |a|+|b|\,.
\]
Continue $\alpha(\xi)$ to negative $\xi$ by zero.
We write
\begin{eqnarray}\label{agk2}
\|k_0(x,0)\|\lesssim \int_0^\infty e^{-\rho|x-\xi|}\alpha(\xi)e^{-\rho \xi}d\xi\le
e^{-\rho x}\int_0^x \alpha(\xi)d\xi+e^{\rho x}
\int_x^\infty \alpha(\xi)e^{-2\rho \xi }d\xi\,.
\end{eqnarray}
Then, using Cauchy-Schwarz inequality, one has
$
\int_0^x \alpha(\xi)d\xi\lesssim (x+x^{1/2})\|A\|_{\rm St}
$. By the change of variable,
\[
e^{\rho x}
\int_x^\infty \alpha(\xi)e^{-2\rho \xi }d\xi=e^{-\rho x}\int_0^\infty e^{-2\rho \eta}\alpha(x+\eta)d\eta\,.
\]
We have
\begin{eqnarray*}
\int_0^\infty e^{-2\rho \eta}\alpha(x+\eta)d\eta=\int_0^1 e^{-2\rho \eta}\alpha(x+\eta)d\eta+\sum_{j=1}^\infty \int_j^{j+1} e^{-2\rho \eta}\alpha(x+\eta)d\eta\le\\
\left(\int_0^1 e^{-4\rho\eta}d\eta\right)^{1/2}\left(\int_0^1 \alpha^2(x+\eta)d\eta\right)^{1/2}+\sum_{j=1}^\infty e^{-2\rho j}\left(\int_j^{j+1} \alpha^2(x+\eta)d\eta\right)^{1/2}
\lesssim \frac{\|A\|_{\rm St}}{\rho^{1/2}}
\end{eqnarray*}
by virtue of Cauchy-Schwarz inequality.
Summing up, we get
\[
\|k_0(x,0)\|\lesssim (x+x^{1/2}+\rho^{-1/2})e^{-\rho x}\|A\|_{\rm St}\lesssim \frac{e^{-\rho x/2}\|A\|_{\rm St}}{\rho^{1/2}}\,.
\]
The Stummel condition is translation-invariant on the line which implies \eqref{agk1} for $j=0$:
\begin{equation}\label{sd_ll}
\|k_0(x,y)\|\le C \frac{e^{-\rho|x-y|/2}}{\rho^{1/2}}\|A\|_{\rm St}\,.
\end{equation}
We can write
$
k_{j+1}(x,y)=\int_{\R^+}G^0(x,\xi,i\rho)V(\xi)k_j(\xi,y)d\xi
$ and use inductive assumption to conclude that
\[
\|k_{j+1}(x,y)\|\le C_1\int_{\R^+}e^{-\rho|x-\xi|}\alpha(\xi)\cdot \|k_j(\xi,y)\|d\xi\le \frac{C_1C^{j+1}\|A\|^{j+1}_{\rm St}}{\rho^{(j+1)/2}}\int_{\R^+}e^{-\rho|x-\xi|}\alpha(\xi)e^{-\rho|\xi-y|/2}d\xi\,.
\]
For $y=0$, we get
\begin{equation}\label{agk4}
\int_{\R^+}e^{-\rho|x-\xi|}\alpha(\xi)e^{-\rho \xi/2}d\xi= e^{-\rho x/2}\cdot e^{-\rho x/2}\int_0^x e^{\rho\xi/2}\alpha(\xi)d\xi+e^{\rho x}\int_x^\infty \alpha(\xi)e^{-3\rho \xi/2}d\xi\,.
\end{equation}
Then, we write
\[
e^{-\rho x/2}\int_0^x e^{\rho\xi/2}\alpha(\xi)d\xi=\int_0^x e^{-\rho \eta/2}\alpha(x-\eta)d\eta \le \int_0^1 e^{-\rho \eta/2}\alpha(x-\eta)d\eta+\sum_{j=1}^{\infty}\int_j^{j+1} e^{-\rho \eta/2}\alpha(x-\eta)d\eta\le
\]
\[
\left(\int_0^1 e^{-\rho \eta}d\eta\right)^{1/2}\left(\int_0^1\alpha^2(x-\eta)d\eta\right)^{1/2}+\sum_{j=1}^{\infty}e^{-\rho j/2}\left(\int_j^{j+1} \alpha^2(x-\eta)d\eta\right)^{1/2}\lesssim \frac{\|A\|_{\rm St}}{\rho^{1/2}}\,.
\]
Estimating the second integral in \eqref{agk4} in a similar way, we have
\[
\int_{\R^+}e^{-\rho|x-\xi|}\alpha(\xi)e^{-\rho \xi/2}d\xi\le C_2 \frac{e^{-\rho x/2}\|A\|_{\rm St}}{\rho^{1/2}}
\]
and, using translation invariance of Stummel condition,
\[
\int_{\R^+}e^{-\rho|x-\xi|}\alpha(\xi)e^{-\rho|\xi-y|/2}d\xi\le C_2 \frac{e^{-\rho |x-y|/2}\|A\|_{\rm St}}{\rho^{1/2}}\,.
\]
Thus, 
\[
\|k_{j+1}(x,y)\|\le \frac{C_1C_2C^{j+1}e^{-\rho|x-y|/2}\|A\|^{j+2}_{\rm St}}{\rho^{(j+2)/2}}\,.
\]
Choosing $C$ sufficiently large, e.g., larger than $C_1C_2$, we show \eqref{agk1} for $j+1$. This proves the claim. 
Now, \eqref{agk6} implies $\|G(x,y,i\rho)\|\lesssim e^{-\rho|x-y|/2}$ provided $\rho=2C(1+ \|A\|^2_{\rm St})$. Thus, \eqref{agk9} finishes the proof.
\qed

\begin{lemma}\label{a6}
Let $h\in L^2(\mathbb{R})$. Then,
\begin{equation}\label{sd_6}
\lim_{t\to+\infty}\left\|e^{it\partial^2_{xx}}h- \frac{1}{1+i}\frac{e^{ix^2/(4t)}}{\sqrt t}\widehat h(x/(2t))\right\|_{L^2(\R)}=0, 
\end{equation}
and, taking inverse Fourier transform,
\begin{equation}\label{sd_7}
\lim_{t\to +\infty}\left\|  \frac{1}{1+i}\left(\frac{e^{ix^2/(4t)}}{\sqrt t}\widehat h(x/(2t))\right)^{\vee}-e^{-it\xi^2}\ \widecheck{h}(\xi)\right\|_{L^2(\R)}= 0\,.
\end{equation}
Suppose $\widehat h\in C_c^\infty(\mathbb{R})$,  then
\begin{equation}\label{ff1}
\sup_{t>1,\alpha,\beta\in \R}\left\|\int_{\alpha t}^{\beta t}\frac{e^{ix^2/(4t)}}{\sqrt t}\widehat h(x/(2t))e^{ixk}dx\right\|_{L^\infty(\mathbb{R})}<C_{(h)}\,.
\end{equation}
\end{lemma}
\begin{proof}
Formula \eqref{sd_6} can be found in \cite{ya} (see formulas (4.10) and (4.12) there). Then, \eqref{sd_7} is a direct corollary. Proof of \eqref{ff1} follows from a direct calculation:
\[
\int_{\alpha t}^{\beta t}\frac{e^{ix^2/(4t)}}{\sqrt t}\widehat h(x/(2t))e^{ixk}dx=\frac{e^{-itk^2}}{\sqrt t}\int_{\alpha t}^{\beta t}\exp\left(i\left(    \frac{x}{2\sqrt t}+k\sqrt t      \right)^2    \right)\widehat h(x/(2t))dx
\]
\[
=2e^{-itk^2}\int_{\sqrt t(0.5\alpha+k)}^{\sqrt t(0.5\beta+k)}\exp(i\xi^2)\widehat h(-k+\xi/\sqrt t)d\xi\,.
\]
Now, consider an integral
\[
\int_{0}^{l}\exp(i\xi^2)\widehat h(-k+\xi/\sqrt t)d\xi
\]
for arbitrary $l\in \R,k\in \R,t\ge 1$ and let $\mu_0$ be a bump function introduced in \eqref{l3}. We have
\[
\int_{0}^{l}\exp(i\xi^2)\widehat h(-k+\xi/\sqrt t)d\xi=\int_{0}^{l}\exp(i\xi^2)\widehat h(-k+\xi/\sqrt t)\mu_0d\xi+\int_{0}^{l}\exp(i\xi^2)\widehat h(-k+\xi/\sqrt t)(1-\mu_0)d\xi\,.
\]
The first integral is bounded uniformly in all parameters since $\widehat h\in C^\infty_c(\R)$. For the second one, we can write
\begin{eqnarray*}
\int_{0}^{l}\exp(i\xi^2)\widehat h(-k+\xi/\sqrt t)(1-\mu_0)d\xi=\int_{0}^{l}\Bigl(\exp(i\xi^2)\Bigr)'\frac{\widehat h(-k+\xi/\sqrt t)(1-\mu_0)}{2i\xi}d\xi\\=
\exp(il^2)\frac{\widehat h(-k+l/\sqrt t)(1-\mu_0(l))}{2il}-\int_0^l \exp(i\xi^2)\left(\frac{\widehat h(-k+\xi/\sqrt t)(1-\mu_0(\xi))}{2i\xi}\right)'d\xi\,.
\end{eqnarray*}
The first term is uniformly bounded because $1-\mu_0(0)=0$. For the second one, we can show that each resulting integral is uniformly bounded, e.g., 
\[
\left|\frac{1}{\sqrt t}\int_0^l \exp(i\xi^2)\frac{\widehat h'(-k+\xi/\sqrt t)(1-\mu_0)}{2i\xi}d\xi\right|\lesssim \frac{1}{\sqrt t}\int_0^l |\widehat h'(-k+\xi/\sqrt t)|d\xi\le \|\widehat h'\|_1\,,
\]
\[
\left|\int_0^l \exp(i\xi^2)\frac{\widehat h(-k+\xi/\sqrt t)(1-\mu_0(\xi))}{\xi^2}d\xi\right|\le \|\widehat h\|_\infty\int_0^l \frac{|1-\mu_0(\xi)|}{\xi^2}d\xi\lesssim \|\widehat h\|_\infty\,,
\]
\[
\left|\int_0^l \exp(i\xi^2)\frac{\widehat h(-k+\xi/\sqrt t)\mu'_0(\xi)}{2i\xi}d\xi\right|\lesssim \|\widehat h\|_\infty
\]
and \eqref{ff1} is proved.
\end{proof}

\begin{lemma}\label{a1} Let $\epsilon \in (0,1),\nu>0,a>0$, and $|a\epsilon|\le \nu$. We have
\[
\left|\int_{0}^a e^{iu^2}g(u\epsilon)du\right|\le C_{(g,\nu)} \epsilon
\]
provided that $g\in C^\infty(\mathbb{R})$  and $g(0)=0$.
\end{lemma}
\begin{proof}We have
\begin{equation}\label{sd_0}
\int_{0}^a (e^{iu^2})'g(u\epsilon)u^{-1}du=e^{ia^2}\epsilon \left(\frac{g(a\epsilon)}{a\epsilon}\right)-\epsilon g'(0)-\epsilon\int_{0}^a e^{iu^2}\left(\frac{g(u\epsilon)}{\epsilon u}\right)'du\,.
\end{equation}
We can write $|g(\xi)|\le C_{(g,\nu)}|\xi|$ for $\xi\in [-\nu,\nu]$ and the first term is controlled by $C_{(g,\nu)}\epsilon$ since $|a\epsilon|\le\nu$. For the third one, we  introduce $G(u)\df (g(u)/u)'\in C^\infty(\mathbb{R})$ and write
\[
G_1(u)\df G(u)-G(0), \quad G(u)=G(0)+G_1(u)
\]
so that
\[
\int_0^a e^{iu^2}G(u\epsilon)du=G(0)\int_0^a e^{iu^2}du+\int_0^a e^{iu^2}G_1(\epsilon u)du\,.
\]
The absolute value of the first term is bounded by $C_{(g)}$ uniformly in $a$. For the second one, we can iterate the argument since $G_1\in C^\infty(\mathbb{R})$ and $G_1(0)=0$. We get
\begin{eqnarray}\nonumber
\int_0^a e^{iu^2}G_1(\epsilon u)du=-0.5 i\epsilon \int_0^a (e^{iu^2})'\frac{G_1(\epsilon u)}{\epsilon u}du\\
=-0.5i\epsilon\left(e^{ia^2}\frac{G_1(\epsilon a)}{\epsilon a}- G'_1(0)-\int_0^a e^{iu^2} \left(\frac{G_1(\epsilon u)}{\epsilon u}\right)'du\right)\label{sd_2}\,.
\end{eqnarray}
Writing a rough estimate 
\[
\left|\int_0^a e^{iu^2} \left(\frac{G_1(\epsilon u)}{\epsilon u}\right)'du\right|\le C_{(g)}|a|
\]
and substituting it into \eqref{sd_2} gives 
\[
\left|\int_0^a e^{iu^2}G_1(\epsilon u)du\right|\le C_{(g)}(\epsilon+\epsilon |a|)=C_{(g)}(\epsilon+\nu)\,.
\]
We bring it to \eqref{sd_0} to finish the proof of the Lemma.
\end{proof}

Consider $H$ defined as 
\[
H(x)=\int_{x}^\infty e^{it^2}dt\,.
\]
This integral can be related to the so-called ${\rm erf}$--function whose properties are well-known. However, our purpose is to obtain a specific representation for $H$ for $x\in [1,\infty)$ and we proceed directly as follows.
We change variables and iteratively  integrate by parts $n$ times  to get
\[
H(x)=i\frac{e^{ix^2}}{2x}-\frac{i}{2}\int_{x^2}^\infty \frac{e^{iu}}{u^{3/2}}du=e^{ix^2}\left(\sum_{j=0}^{n-1}\frac{c_j}{x^{1+2j}}+c_n'e^{-ix^2}\int_{x^2}^\infty \frac{e^{iu}}{u^{n+1/2}}du\right)\df e^{ix^2}(H_{1,n}+H_{2,n})
\]
where $\{c_j\}$  and $c_n'$ are some constants.
Let $\mu_+$ be the cutoff function that satisfies conditions: $ \mu_+$ is supported on $(1,\infty)$, $\mu_+(x)=1$ for $x>2$, $\mu_+\in C^\infty(\mathbb{R})$. Define
\[
H_{1,n}^{(m)}\df H_{1,n}\mu_+, \quad H_{2,n}^{(m)}\df H_{2,n}\mu_+\,.
\]
 \begin{lemma}\label{a2}
Let $n>1$. We have $\widehat{H_{1,n}^{(m)}}\in L^1(\mathbb{R})$, $\widehat{H_{2,n}^{(m)}}\in L^1(\mathbb{R})$.
\end{lemma}
\begin{proof}   Consider $H_{2,n}^{(m)}$ first. We have 
\[
|H_{2,n}^{(m)}|\le C_n(1+|x|)^{-(2n+1)}, \quad |\partial_x H_{2,n}^{(m)}|\le C_n(1+|x|)^{-(2n)}, \quad \quad |\partial_{xx}^2 H_{2,n}^{(m)}|\le C_n(1+|x|)^{-(2n-1)}\,.
\]
Therefore, 
\[
|\widehat{H_{2,n}^{(m)}}(\xi)|<C_n(1+|\xi|)^{-2}
\]
and hence $\widehat{H_{2,n}^{(m)}}\in L^1(\mathbb{R})$.
For $H_{1,n}^{(m)}$, consider the first term,
$
x^{-1}\mu_+
$. Other terms can be handled similarly. We have $x^{-1}\mu_+\in C^\infty(\mathbb{R})\cap L^2(\R)$ and all of its derivatives are in $L^2(\mathbb{R})$. Thus, $
\xi^{j}\widehat{(x^{-1}\mu_+)}\in L^2(\mathbb{R})
$ for all $j\in \mathbb{Z}^+$. Therefore, $\widehat{(x^{-1}\mu_+)}(\xi)\in L^1(|\xi|>1)$. For $|\xi|<1$, we can write an estimate
\[
|\widehat{x^{-1}\mu_+}|<C|\log \xi|\,,
\]
which can be verified directly:
\[
\int_1^\infty \frac{\mu_+(x)}{x}e^{- i \xi x}dx=\int_2^\infty \frac{e^{- i \xi x}}{x}dx+O(1)\,.
\]
For $\xi\in (0,1)$,
\[
\int_2^\infty \frac{e^{- i \xi x}}{x}dx=\int_{2\xi}^\infty\frac{e^{- i u}}{u}du=\int_{2\xi}^1\frac{e^{- i u}}{u}du+\int_{1}^\infty\frac{e^{- i u}}{u}du=O(|\log \xi |+1)\,.
\]
For $\xi\in (-1,0)$, the argument is analogous and we get the statement of the Lemma.
\end{proof}

\end{document}